\documentclass[11pt]{amsproc}
\usepackage{amsfonts}
\usepackage{amsmath,amstext,amsbsy,amssymb}

\newtheorem{theorem}{Theorem}[section]
\newtheorem{lemma}[theorem]{Lemma}

\newtheorem{proposition}[theorem]{Proposition}
\newtheorem{corollary}[theorem]{Corollary}

\theoremstyle{definition}
\newtheorem{definition}[theorem]{Definition}

\theoremstyle{remark}
\newtheorem{remark}[theorem]{Remark}

\numberwithin{equation}{section}

\newcommand{\la}{\lambda}
\newcommand{\al}{\alpha}
\begin{document}

\title{Macdonald symmetric functions of rectangular shapes}
\author{Tommy Wuxing Cai}
\address{School of Sciences,
South China University of Technology, Guangzhou 510640, China}
\email{caiwx@scut.edu.cn}


\begin{abstract}Using vertex operator we study Macdonald symmetric functions of rectangular shapes and their
 connection with the q-Dyson Laurent polynomial. We find a vertex operator realization of Macdonald functions
 and thus give a generalized Frobenius formula for them. As byproducts of the realization, we find a q-Dyson constant
  term orthogonality relation which generalizes a conjecture due to Kadell in 2000, and we generalize Matsumoto's hyperdeterminant
  formula for rectangular Jack functions to Macdonald functions.
\end{abstract}
\maketitle

\section{Introduction}
Macdonald symmetric functions \cite{M} $Q_\la(q,t)$ form a remarkable class of symmetric functions.
 They are indexed by partitions $\la$ and depending on parameters $q,t$.  They includes many types of symmetric functions
 as special cases such as Jack functions $Q_\la(\al)$ and Hall--Littlewood functions $Q_\la(0,t)$ . When the Young diagram
 of $\la$ is of rectangular shape, i.e., $\la=(k,k,\cdots,k)$, we say that the corresponding symmetric function is of rectangular
  shape. For example, $Q_{(2,2,2)}(0,t)$ is a rectangular Hall-Littlewood function. We study the rectangular Macdonald functions using vertex operator.

First, we find a solution to the long-standing problem of vertex operator realization of Macdonald symmetric
 functions. Jing's vertex operator \cite{J2} realizes Hall-Littlewood functions gracefully. However, a generalization of
 this realization to Jack functions or Macdonald functions is far from immediate. Jing and the author gave a vertex operator
  realization of Jack functions in \cite{CJ1,CJ3}. In this work, we generalize this realization of Jack functions to Macdonald functions,
   as the limit $q\rightarrow 1$  of the later goes back to the case of Jack functions.
  It comes out that our realization here is different from that in \cite{J2}. Roughly speaking, we use two steps to construct Macdonald functions.
   In the first step, we realize Macdonald functions of rectangular shapes, and in the second step, we use these rectangular Macdonald functions
    to construct those of general shapes (Theorem \ref{T:MacFiltration}). We explain briefly the first step here.

   Let $F=\mathbb{Q}(q,t)$ be the field of rational functions in two independent variables $q,t$. The space of symmetric space is
   a free commutative associative algebra $\Lambda$ generated by variables $p_1,p_2,\cdots$ over $F$ (as we are considering Macdonald functions).
    Let $\mathbb{Q}[\frac{1}{2}\mathbb{Z}]$ be the group algebra of $\{e^{n\eta}:n\in\frac{1}{2}\mathbb{Z}\}$ (with multiplication being $e^{m\eta}e^{n\eta}=e^{(m+n)\eta}$).
     Extend $\Lambda$ to $V$ which is defined to be $V=\Lambda\otimes\mathbb{Q}[\frac{1}{2}\mathbb{Z}]$. We consider the case that $t=q^\beta$ where $\beta$ is
     an arbitrary positive integer. Define the following operators $X_{i}$'s on $V$ using the following generating function:
  \begin{align}
X(z)&=\exp\Big(\sum_{n\geq1}\frac{p_nz^n}{n}\frac{1-q^{n\beta}}{1-q^n}\Big)
M(z)\exp\Big(\sum_{n\geq1}-\frac{\partial}{\partial p_n}z^{-n}\frac{q^{-n\beta}-q^{n\beta}}{1-q^{n\beta}}\Big)\\\nonumber
&=\sum_nX_{-n}z^n,
\end{align}
where $M(z).v\otimes e^{m\eta}=z^{(2m+1)\beta} v\otimes e^{(m+1)\eta}$. For rectangular partition $\la=(k^s)$ and positive integer $\beta$,
 we prove that (as a special case of Theorem \ref{T:main}):
\begin{align}
(X_{-k})^s.1\otimes e^{s\eta/2}=\frac{(q;q)_{s\beta}}{(q;q)_\beta^s}Q_{(k^s)}(q,q^\beta)\otimes e^{s\eta/2},
\end{align}
where the q-Pochhammer symbol $(z;q)_n=(1-z)(1-q^z)\dotsm(1-q^{n-1}z)$. This way, we realize Macdonald functions of rectangular shapes using vertex operator.

Second, we give a q-Dyson orthogonality relation which generalizes a conjecture of Kadell.
 To study the vertex operator realization of Macdonald functions, we consider the following Laurent polynomial:
\begin{align}\label{F:gqdlp}
F_{\boldsymbol{\beta},q}[s;t]&=F_{\boldsymbol{\beta},q}(z_1,\dotsc,z_s;w_1,\dotsc,w_t)\\\nonumber
&=\prod_{1\leq i< j\leq s}\Big(\frac{z_i}{z_j};q\Big)_{\beta_i}\Big(\frac{qz_j}{z_i};q\Big)_{\beta_j}
\prod_{i=1}^s\prod_{j=1}^t\Big(\frac{z_i}{w_j};q\Big)_{\beta_i}^{-1},
\end{align}
where $\boldsymbol{\beta}=(\beta_1,\dotsc,\beta_s)$ is a sequence of positive integers.
If for all $i,j$, $|q^bz_i/w_j|<1$
($b=0,1,\dotsc,\beta_i-1$), then we can expand $F_{\boldsymbol{\beta},q}[s;t]$
into an infinite sum of monomials of the form $z_1^{k_1}\dotsm z_s^{k_s}w_1^{-m_1}\dotsm
w_t^{-m_t}$ with $k_i\in \mathbb{Z}$,
$m_j\in \mathbb{Z}_{\geq0}$ and $k_1+\dotsb+k_s=m_1+\dotsb+m_t$. However, not all monomials of this form appear in the expansion.
The orthogonality relation (Theorem \ref{T:Fbetaqst}) asserts the vanishing of certain monomials, with the coefficients of some monomials also given.
  The $t=1$ case of this relation (see Theorem \ref{C:Kadell}) was conjectured by Kadell \cite{K}.
   Andrews's q-Dyson constant term conjecture \cite{A}, which was first proved by Zeilberger and Bressoud \cite{Z}, is about $F_{\boldsymbol{\beta},q}[s;0]$.
    We also give a proof of this conjecture, based on the observation that $F_{\boldsymbol{\beta},q}[s;0]$ and $F_{\boldsymbol{\beta},q}[s;1]$ have the same constant term.
    The splitting formula for $F_{\boldsymbol{\beta},q}[s;1]$ (see Proposition \ref{P:Fbetaqstw1}) is critical in these proofs.
    We like to mention that this splitting formula may be helpful in the computation of the non-constant term of $F_{\boldsymbol{\beta},q}[s;0]$.
    After this study was finished, we found that the first proof of Kadell's conjecture was given in 2012 \cite{KLW}. Our proofs here are elementary.

Third, we find a formula which connects q-Dyson Laurent polynomial with Macdonald functions of rectangular shapes. We have the following raising-operator-type
formula for rectangular Macdonald functions:
\begin{align}\label{F:reclowing}
\frac{(q;q)_{s\beta}}{(q;q)_\beta^s}Q_{(k^s)}(q,q^\beta)
=\prod_{1\leq<i<j\leq s}\big(\frac{D_i}{D_j};q\Big)_\beta\big(\frac{qD_j}{D_i};q\Big)_\beta.(Q_k(q,q^\beta))^s,
\end{align}
where $Q_n$'s are the one-row Macdonald functions and the operator $D_i$ lowers the $i$th subscript
 by $1$: $D_i.Q_{\la_1}\cdots Q_{\la_i}\dotsm Q_{\la_s}=Q_{\la_1}\dotsm Q_{\la_i-1}\dotsm Q_{\la_s}$ (see Definition \ref{D:Di} for rigorous definition).
  Formula (\ref{F:reclowing}) is a special case-the $t=0$ case-of the formula (\ref{F:qDysonMac}) in Theorem \ref{T:main}.
The $q=1$ case of (\ref{F:reclowing}) is equivalent to Matsumoto's Hyperdeterminant formula for rectangular Jack functions
\cite{Ma}. Matsumoto's formula was generalized to almost rectangular shapes ($Q_\la$ with $\la=((k+1)^t,k^s)$) in \cite{BBL},
and this generalized formula is equivalent to the $q=1$ case of our formula (\ref{F:qDysonMac}).
Formula (\ref{F:reclowing}) can also be specified to the case of Hall--Littlewood functions.
To do this, we set $t=q^\beta$, then use the formula $(z;q)_n=(z;q)_\infty/(zq^n;q)_\infty$ (with $(z;q)_\infty=(1-z)(1-qz)(1-q^2z)\dotsm$). Formula (\ref{F:reclowing}) turns into:
\begin{align}
\label{F:reclowingtcase}
&\frac{(q;q)_\infty}{(qt^s;q)_\infty} \Big(\frac{(qt;q)_\infty}{(q;q)_\infty}\Big)^s Q_{(k^s)}(q,t)\\\nonumber
&=\prod_{1\leq<i<j\leq s}\frac{(D_i/D_j;q)_\infty}{(tD_i/D_j;q)_\infty}\frac{(qD_j/D_i;q)_\infty}{(qtD_j/D_i;q)_\infty}.(Q_k(q,t))^s.
\end{align}
Now the $q=0$ case of this formula is the rectangular case of the well-known raising operator formula for Hall--Littlewood functions:
\begin{align}
\label{F:raisingHL}
Q_{(k^s)}(q,t)=\prod_{1\leq<i<j\leq s}\frac{(1-D_i/D_j)}{(1-tD_i/D_j)}.(Q_k(q,t))^s.
\end{align}
We can not generalize formula (\ref{F:reclowingtcase}) to general shapes.
However, we can express a Macdonald function as the coefficient of some monomial of a Laurent polynomial (see Remark \ref{R:MacAscoefficient} at the end of the last section).
 We should also mention that the raising operator formula for Macdonald functions was given by  Lassalle and Scholosser \cite{LS},
  as a generalization of the formula for the two-row Macdonald functions \cite{JJ}.

This paper is organized as following. The basic notions about Macdonald functions are given in Section 2, and in Section 3 we
 define the vertex operator which we use to realize Macdonald functions. The q-Dyson orthogonality relations are studied in Section 4.
We devote Section 5 to the vertex operator realization of Macdonald functions of almost rectangular shapes, and the special case of Jack functions are also considered.
 Finally in Section 6, we use the rectangular Macdonald functions to construct Macdonald functions of general shapes and give a generalized Frobenious formula for Macdonald functions.

\section{Partitions and Macdonald functions}
We give some notations and definitions of partitions and Macdonald symmetric
functions in this section, most of which are from \cite{M}.

 A partition is
a sequence $\la=(\la_1,\la_2,\dotsc,\la_s)$ of non-negative integers
in weakly decreasing order. The set of all partitions is denoted by $\mathcal{P}$.
The length $l(\la)$ of $\la$ is
defined to be the number of nonzero parts of $\la$; i.e., $l(\la)=\operatorname{max}\{i:\la_i\neq0\}$. We adopt the convention
 that $\la_i=0$ for $i>l(\la)$. We write $\la\vdash n$ if the weight
of $\la$, which is defined to be $|\la|=\sum_{i\geq1}\la_i$, equals $n$.
For $\la,\mu$ of the same weight, we write $\la\geq\mu$ if
$\sum_{j\leq i}(\la_j-\mu_j)\geq0$ for all $i$. This
defines the dominance ordering. As usual, we write $\la>\mu$ if $\la\geq\mu$ but $\la\neq\mu$.

For a partition $\la$, we let $x^\la$ denote the monomial $x_1^{\la_1}x_2^{\la_2}\dotsm$.
For a Laurent polynomial $F=f(x_1,x_2,\cdot,x_s)$, we use notation $\operatorname{C.T.}F$ for the constant term of $F$.

Sometimes we write $\la=(1^{m_1}2^{m_2}\dotsm)$, where $m_i=m_i(\la)$ is the multiplicity of
 $i$ in $\la$. For two partitions $\la,\mu$, one defines
the set union $\la\cup\mu$ by $m_i(\la\cup\mu)=m_i(\la)+m_i(\mu)$ for
all $i$.

In this paper, we also write a partition in the form of
$\la=(a_1^{n_1}a_2^{n_2}\dotsm a_r^{n_r})$, where it means
$m_{a_j}(\la)=n_j>0$ and $a_1>a_2>\dotsc>a_r>0$. For example, one
has $(4^3)=(4,4,4)$ and $(4^23^1)=(4,4,3)$.

To visualize partitions, one identifies a partition $\la$ with its
Young diagram $Y(\la)$, which is a collection of left justified rows
with $\la_1$ boxes on the first (top) row and $\la_2$ boxes on the
second row and so on. We call a partition with all its parts
identical a rectangular partition, a partition of rectangular
shape or a rectangle, as its Young diagram is of that shape. So $\la=(k^s)$
($k,s>0$) is of rectangular shape. Similarly, one calls
$\la=((k+1)^t,k^s)$ ($k,s\geq1$, $t\geq0$) an almost rectangular partition \cite{BBL}.

Let $\Lambda_F$ be the free commutative associative algebra generated by $p_1,p_2,\dotsc$ over $F$, where $F=\mathbb{Q}(q,t)$ is the
 field of rational functions in two independent indeterminate
$q,t$. We call the elements in $\Lambda_F$ symmetric functions. The power sum symmetric functions $p_\la=p_{\la_1}p_{\la_2}\dotsm$
($\la\in\mathcal {P}$, $p_0=1$) form a basis of $\Lambda_F$. Let $\Lambda_F^n$ be the subspace spanned by the set $\{p_\la:|\la|=n\}$.
We see that $\Lambda_F$ is a graded algebra; it is the direct sum of $\Lambda_F^0=F$, $\Lambda_F^1=Fp_1$, $\Lambda_F^2$ and so on.
 We define a scalar product on $\Lambda_F$ by
\begin{equation}
\langle
p_\lambda,p_\mu\rangle=\delta_{\lambda\mu}z_\la\prod_{i\geq1}\frac{1-q^{\la_i}}{1-t^{\la_i}}
 \hspace{0.5cm}(\lambda,\mu\in\mathcal {P}),
\end{equation}
where $\delta_{\lambda\mu}$ is the Kronecker symbol, and
$z_\la=\prod_{i\geq1}i^{m_i(\la)}m_i(\la)!$. %The Macdonald functions
The generalized complete symmetric function $Q_n(q,t)$ is defined by:
\begin{equation}
\exp\Big(\sum_{n\geq1}\frac{p_nz^n}{n}\frac{1-t^n}{1-q^n}\Big)=\sum_n Q_n(q,t)z^n.
\end{equation}
The generalized complete symmetric functions
$g_\la(q,t)=Q_{\la_1}(q,t)Q_{\la_2}(q,t)\dotsm$
($\la\in\mathcal{P}$) form a basis of $\Lambda_F$.
Macdonald symmetric functions $Q_\la(q,t)$ form an orthogonal basis of $\Lambda_F$. It is given by the following:
\begin{lemma}\cite{M}\label{D:Qlambda}
There is a unique family of symmetric functions $Q_{\la}(q,t)\in \Lambda_F$ ($\la\in\mathcal{P}$) satisfying the following conditions:\\
\textnormal{(1)} $Q_\la(q,t)=g_{\la}(q,t)+\sum_{\mu>\la}d_{\la\mu}g_\mu(q,t)$;\\
\textnormal{(2)} $\langle Q_\la(q,t),g_\mu(q,t)\rangle=0$ if $\mu>\la$.
\end{lemma}
We see that $Q_\la(q,t)$ is in $\Lambda_F^{|\la|}$ and $Q_{(n)}(q,t)$ equals $Q_n(q,t)$.
We remark that our treatment of symmetric functions here is to look at them as elements in a commutative algebra generated by $p_1,p_2,\dotsc$. Originally $p_n$
is the formal power sum $x_1^n+x_2^n+\dotsm$, and the symmetry is in these variables $x_1,x_2,\dotsc$.
 However, we are interested in this commutative algebra treatment of $\Lambda_F$ in this study.
   In fact, Lemma \ref{D:Qlambda} can also be proved-in \cite{CJ2} for example-inside this framework (without going back to those varialbles $x_1,x_2,\dotsc$).
    We should also mention that there is another basis $\{P_\mu(q,t)\}$ such that $\langle Q_\la(q,t),P_\mu(q,t)\rangle=\delta_{\la\mu}$.
Both $Q_\la$'s and $P_\la$'s are called Macdonald symmetric functions but we are only interested in $Q_\la(q,t)$'s. Moreover, we only study
the case that $t=q^\beta$, where $\beta$ is an arbitrary positive integer.

We will need the following formula for the norm of the Macdonald functions.
\begin{lemma}\cite{M}\label{L:Macnorm}
For a partition $\la$, one has
\begin{equation}\label{F:norm}
\langle Q_ \la(q,q^\beta),Q_\la(q,q^\beta)\rangle=\prod \frac{1-q^{\la_i-j+\beta(\la_j'-i+1)}}{1-q^{\la_i-j+1+\beta(\la_j'-i)}},
\end{equation}
where the product go through $(i,j)$ such that $1\leq i\leq l(\la), 1\leq j\leq \la_i$, and $\la'_j$ is the number of those $i$'s such
that $\la_i\geq j$.
\end{lemma}
We will also study Jack functions $Q_\la(\al)$. They form an orthogonal basis with respect to another scalar product. They can be
looked as a special case of Macdonald functions. To put it shortly, $Q_\la(\beta^{-1})$ is the limit of $Q_ \la(q,q^\beta)$ as $q$ goes to $1$.
\section{The vertex operator}

For convenience, we extend the space of symmetric functions $\Lambda_F$ a little. We define the algebra $V=\Lambda_F\otimes
\mathbb{Q}[\frac12\mathbb{Z}]$, where
 $\mathbb{Q}[\frac12\mathbb{Z}]$ is the group algebra of $\{e^{n\eta}|n\in\frac{1}{2}\mathbb{Z}\}$
 with multiplication given by
  $e^{m\eta}e^{n\eta}=e^{(m+n)\eta}$.
   We can extend the scalar product to $V$ by
\begin{equation}
\langle u\otimes e^{m\eta},v\otimes e^{n\eta}\rangle=\langle u,v\rangle\delta_{mn}.
\end{equation}
Now let us define the following vertex operator on the vector space $V$:
\begin{align}\label{D:vertexoperator}
X(z)&=\exp\Big(\sum_{n\geq1}\frac{h_{-n}z^n}{n}\frac{1-q^{n\beta}}{1-q^n}\Big)
M(z)\exp\Big(\sum_{n\geq1}\frac{h_{n}z^{-n}}{-n}\frac{q^{-n\beta}-q^{n\beta}}{1-q^n}\Big)\\\nonumber
&=\sum_nX_{-n}z^n,
\end{align}
where for positive integer $n$ and $v\otimes e^{m\eta}\in V$,
\begin{align*}
&h_{-n}.v\otimes e^{m\eta}=p_{n}v\otimes e^{m\eta},\\
&h_n.v\otimes e^{m\eta}=n\frac{1-q^n}{1-q^{n\beta}}\frac{\partial}{\partial p_{n}}v\otimes e^{m\eta},\\
&M(z).v\otimes e^{m\eta}=z^{(2m+1)\beta} v\otimes e^{(m+1)\eta}.
\end{align*}
\begin{remark}
By definition, the operators $h_n$ and $h_{-n}$ are conjugate; i.e., for $x,y\in V$, one has
$\langle h_n.x,y\rangle=\langle x,h_{-n}.y\rangle$.
\end{remark}
We call $M(z)$ the middle term of the vertex operator $X(z)$.
The term to the left (resp. right) side of $M(z)$ is called the creation (resp. annihilation) part of $X(z)$, as $h_n$ maps $\Lambda_F^m\otimes e^{s\eta}$ into $\Lambda_F^{m-n}\otimes e^{s\eta}$ (for $n\neq0$, and we set $\Lambda_F^{k}=0$ if $k<0$).
Acting on $v\otimes e^{m\eta}$, the effect the middle term is independent of $v$ and
 that of $h_n$ ($n\neq0$) is independent of $e^{m\eta}$.
 We thus say that $M(z)$ acts on the group algebra part of the tensor product,
 while the creation part and annihilation part of $X(z)$ act on the symmetric function part.

We caution the reader that our notation $h_i$'s
has nothing to do with the complete symmetric polynomials, which are
one-row Schur polynomials and we will use the notation $S_n$ for them. Restricting their actions on $V_m=\Lambda_F\otimes e^{m\eta}$ ($m\in\frac 12\mathbb{Z}$),
 these $h_i$'s together with the identity map $I$ form a basis of a Heisenberg algebra (as a Lie sub-algebra of $\mathrm{End}_F(V_m)$), with the Lie bracket satisfying
 \begin{equation}\label{F:lie}
 [h_n,h_{-m}]=\delta_{m,n}m\frac{1-q^n}{1-q^{n\beta}}I\hspace{0.5cm} ( n>0 ).
 \end{equation}

This vertex operator is related to the Laurent polynomial $F_{\beta,q}[s;t]$.
 Recall that it is the special case of $F_{\boldsymbol{\beta},q}[s;t]$ (see (\ref{F:gqdlp})) when all $\beta_i$'s equal $\beta$.  The special case $F_{\beta,q}[s;0]$ is equal to
\begin{align}
 F_{\beta,q}(z_1,\dotsc,z_s)
=\prod_{1\leq i< j\leq s}\Big(\frac{z_i}{z_j};q\Big)_{\beta}\Big(\frac{qz_j}{z_i};q\Big)_{\beta}.
\end{align}
Define the lowering operator $D_i$'s as the following:
\begin{definition}\label{D:Di}

Let $F=\sum_{i_1,\dotsc,i_s}c_{i_1i_2\dotsc i_s}D_1^{i_1}\dotsm
D_s^{i_s}$ be a Laurent polynomial. For a symmetric functions $g_\la=Q_{\la_1}\dotsm
Q_{\la_s}$, the action of $F$ on $g_\la$ is defined by
\begin{equation}
F.g_\la=\sum_{i_1,\dotsc,i_s}c_{i_1i_2\dotsc i_s}Q_{\la_1-i_1}\dotsm Q_{\la_s-i_s}.
\end{equation}
\end{definition}
We see that $D_i$ lowers the $i$th subscript by one, and this is why we call it a lowering operator.
 Note that if $F$ is a Laurent polynomial of those $D_j/D_i$'s with $i<j$, then $F$ is the usual raising operator.
 Thus lowering operator is a generalization of raising operator. Also recall that we let $\operatorname{C.T.}F$ denote the constant term of $F$ and $x^\la=x_1^{\la_1}x_2^{\la_2}\dotsm$.
 Now we can state the connection between $X(z)$ and $F_{\beta,q}[s;t]$.
\begin{lemma}\label{L:Xlambda}
 For two partitions $\la$ and $\mu$ of length $s$ and $t$ respectively, we have
\begin{align}\label{F:Xlambdaraising}
&X_{-\la_1}\dotsm X_{-\la_s}.1\otimes e^{-s\eta/2}\\
\nonumber&=\epsilon_q(\beta,s)F_{\beta,q}(D_1,\dotsc,D_s).
Q_{\la_1}(q,q^\beta)\dotsc Q_{\la_s}(q,q^\beta) \otimes e^{s\eta/2},\\
\label{F:XYC.T.}
 &\langle X_{-\la_1}\dotsm X_{-\la_s}.1\otimes e^{-s\eta/2},
Q_{\mu_1}(q,q^{\beta})\dotsm Q_{\mu_t}(q,q^{\beta}) \otimes
e^{s\eta/2}\rangle\\ \nonumber
&=\epsilon_q(\beta,s)\operatorname{C.T.} (z^{\la}/w^{\mu})^{-1}F_{\beta,q}[s;t],
\end{align}
where
\begin{equation}
\epsilon_q(\beta,s)=(-1)^{\beta
s(s-1)/2}q^{-\beta(\beta+1)s(s-1)/4}.
\end{equation}

\end{lemma}
\begin{proof}
For convenience, set
\begin{align*}
B_j=\sum_{n\geq1}\frac{h_{n}z_j^{-n}}{-n}\frac{q^{-n\beta}-q^{n\beta}}{1-q^n},
~~~A_i=\sum_{n\geq1}\frac{h_{-n}z_i^n}{n}\frac{1-q^{n\beta}}{1-q^n}.
\end{align*}
Then we can write $X(z_k)=e^{A_k}M(z_k)e^{B_k}$.
Using (\ref{F:lie}), one has
\begin{align*}
[B_j,A_i]&=\sum_{n\geq1}\frac{(z_i/z_j)^n}{-n}(q^{n(-\beta)}+q^{n(-\beta+1)}+q^{n(-\beta+2)}+\dotsb+q^{n(\beta-1)}).\\
\end{align*}
Thus the exponential of it can be written as following:
\begin{align*}
e^{[B_j,A_i]}&=\prod_{k=-\beta}^{\beta-1}\exp\Big(\sum_{n\geq1}\frac{(z_i/z_j)^n}{-n}q^{nk}\Big)
=\prod_{k=-\beta}^{\beta-1}(1-q^kz_i/z_j).
\end{align*}
As $[B_j,A_i]$ commutes with $B_j$ and $A_i$ (as operators on $V$), we know that
$e^{B_j}e^{A_i}$ equals $e^{A_i}e^{B_j}e^{[B_j,A_i]}$. This enables us to move the $e^{B_j}$'s to the right side of $e^{A_i}$'s as in the following:
\begin{align*}
&X(z_s)\dotsm X(z_2)X(z_1)\\
&=e^{A_s}\dotsm e^{A_1}e^{B_s}\dotsm
e^{B_1}M(z_s)\dotsm M(z_1) \prod_{1\leq
i<j\leq s}\prod_{k=-\beta}^{\beta-1}\Big(1-q^k\frac{z_i}{z_j}\Big).
\end{align*}
Notice that $e^{B_k}.1\otimes e^{m\eta}$ equals $1\otimes e^{m\eta}$ and the action of $e^{A_k}$ is the same as the multiplication of $\sum_{n\geq0}Q_n(q,q^\beta)z_k^n$ (to the symmetric function part). Hence we have
\begin{align*}
&X(z_s)\dotsm X(z_2)X(z_1).1\otimes e^{-s\eta/2}\\\nonumber
&=z_1^{(-s+1)\beta}z_2^{(-s+3)\beta} \dotsm z_s^{(s-1)\beta}
\prod_{1\leq i<j\leq
s}\prod_{k=-\beta}^{\beta-1}\Big(1-q^k\frac{z_i}{z_j}\Big)e^{A_s}\dotsm
e^{A_1}.1\otimes e^{s\eta/2}\\\nonumber
&=\epsilon_q(\beta,s)\prod_{1\leq i\leq j\leq
s}\Big(\frac{z_i}{z_j};q\Big)_\beta\Big(\frac{q z_j}{z_i};q\Big)_\beta e^{A_s}\dotsm e^{A_1}.1
\otimes e^{s\eta/2}\\\nonumber
&=\epsilon_q(\beta,s)\prod_{1\leq i\leq
j\leq s}\Big(\frac{z_i}{z_j};q\Big)_\beta\Big(\frac{q z_j}{z_i};q\Big)_\beta
\prod_{k=1}^s\Big(\sum_{n\geq0}Q_n(q,q^\beta)z_k^n\Big)\otimes
e^{s\eta/2}.
\end{align*}
Now we have the formula (\ref{F:Xlambdaraising}) (by Definition (\ref{D:Di})).

To prove the second formula, first we set
\begin{align*}
&Y(w)=\exp\Big(\sum_{n\geq1}\frac{h_{-n}w^{-n}}{n}\frac{1-q^{n\beta}}{1-q^n}\Big)=\sum_nY_{-n}w^{-n},\\
&Y(w)^*=\exp\Big(\sum_{n\geq1}\frac{h_{n}w^{-n}}{n}\frac{1-q^{n\beta}}{1-q^n}\Big)=\sum_nY_{-n}^*w^{-n}.
\end{align*}
Note that $Y(w)$ is the creation part of $X(w)$. The action of $Y_{-n}$ is the same as multiplication of $Q_n(q,q^\beta)$ (to the symmetric functions part)
 and $\langle x,
Y_{-n}.y\rangle=\langle Y_{-n}^*.x,y\rangle$, i.e.,$Y_{-n}^*$ is really the conjugate of $Y_{-n}$.
 The left
side of (\ref{F:XYC.T.}) is equal to the coefficient of
$z^{\la}/w^{\mu}=z_1^{\la_1}z_2^{\la_2}\dotsm/(w_1^{\mu_1}w_2^{\mu_2}\dotsm)$ in the following
\begin{align}
&\langle X(z_s)\dotsm X(z_1).1\otimes e^{-s\eta/2},Y(w_1)\dotsm
Y(w_t).1\otimes e^{s\eta/2}\rangle\\\nonumber
 &=\langle Y(w_t)^*\dotsm
Y(w_1)^*X(z_s)\dotsm X(z_1).1\otimes e^{-s\eta/2},1\otimes
e^{s\eta/2}\rangle\\\nonumber
 &=\epsilon_q(\beta,s)\prod_{1\leq i<
j\leq s}\Big(\frac{z_i}{z_j};q\Big)_\beta\Big(\frac{qz_j}{z_i};q\Big)_\beta
\prod_{i=1}^s\prod_{j=1}^t\Big(\frac{z_i}{w_j};q\Big)_\beta^{-1}.
\end{align}
Now the formula (\ref{F:XYC.T.}) follows.
\end{proof}
Let us use the vertex operator $Y(w)$ and its conjugate $Y(w)^*$ to analysis the reciprocal of the q-Pochhammer symbol $(z;q)_\beta$.
\begin{lemma}\label{C:cozqbeta1}
For non-negative integers $m,\beta$, the coefficient of $z^m$ in $(z;q)_\beta^{-1}$ is:
\begin{align*}
\operatorname{C.T.} z^{-m}(z;q)_\beta^{-1}&=(q^\beta;q)_m/(q;q)_m.
\end{align*}
\end{lemma}
\begin{proof}
On one hand, it is known (from Lemma \ref{L:Macnorm}) that
\begin{equation}
\langle Q_m(q,q^\beta), Q_m(q,q^\beta)\rangle=(q^\beta;q)_m/(q;q)_m.
\end{equation}
On the other hand, $\langle Q_m(q,q^\beta), Q_m(q,q^\beta)\rangle$ is the coefficient of $(w_2/w_1)^{m}$ in
$\langle Y(w_1).1, Y(w_2^{-1}).1\rangle$; i.e.,
\begin{align*}
&\langle Q_m(q,q^\beta), Q_m(q,q^\beta)\rangle=\operatorname{C.T.} (w_2/w_1)^{-m}\langle Y(w_1).1, Y(w_2^{-1}).1\rangle\\
&=\operatorname{C.T.}(w_2/w_1)^{-m} \langle Y(w_2^{-1})^*Y(w_1).1,1\rangle=\operatorname{C.T.}(w_2/w_1)^{-m}(w_2/w_1;q)_\beta^{-1}.
\end{align*}
Comparing the two expressions finishes the proof.
\end{proof}

\section{A q-Dyson constant term relation}
In this section we study the properties of the Laurent polynomial $F_{\boldsymbol{\beta},q}[s;t]$. Some of these properties will be applied in the next section to the realization of Macdonald functions, and some generalizes Kadell's conjecture. We like to mention that all the proofs are \emph{elementary} and this section can be read independently.

For latter application, let us first analysis the reciprocal of the q-Pochhammer symbol $(z;q)_\beta$ further.
\begin{lemma}\label{C:cozqbeta}
Assume that $m,\beta$ are non-negative integers and $q\neq 0,1$ is a complex number.
We have the following three expressions for the coefficient of $z^m$ in $(z;q)_\beta^{-1}$:
\begin{align*}
\operatorname{C.T.} z^{-m}(z;q)_\beta^{-1}&=S_m(1,q,\dotsc,q^{\beta-1})\\
                           &=\sum_{b=0}^{\beta-1}q^{bm}(q^{-b};q)_b^{-1}(q;q)_{\beta-1-b}^{-1}\\
                           &=(q^\beta;q)_m/(q;q)_m,
\end{align*}
where $S_n(x_1,\dotsc,x_k)$ is the one-row Schur polynomial or homogeneous complete symmetric polynomial, and it is required that $(\beta,m)\neq(0,0)$ for the second expression.
\end{lemma}

\begin{proof}
The first expression is direct by the definition of complete homogeneous symmetric polynomials (which are one-row Schur polynomials).

The second expression comes from the following identity of complex (or real) functions in $z$:
\begin{equation}\label{E:zwidentity}
(z;q)_\beta^{-1}=\sum_{b=0}^{\beta-1}(1-q^bz)^{-1}(q^{-b};q)_b^{-1}(q;q)_{\beta-b-1}^{-1}.
\end{equation}
To prove this identity, one multiplies both sides by $(z;q)_\beta$; then the right side turns into a sum of $\beta$ polynomials in $z$, each of which is of degree $\beta-1$.  Now evaluate $z=q^{-b}$ for
$b=0,1,\dotsc,\beta-1$. One always gets $1=1$ and thus proves the identity.

The third expression is given by Lemma \ref{C:cozqbeta1}. We would like to give an \emph{elementary} proof without using Macdonald functions and the vertex operator. Let us make induction on $\beta$. It is trivial when $\beta=0$.
 Assume that it is true for $\beta\geq0$. Notice that $(z;q)_{\beta+1}^{-1}$ is the product of $(z;q)_{\beta}^{-1}$ and $(1-q^\beta z)^{-1}$; the coefficient of $z^m$ in $(z;q)_{\beta+1}^{-1}$ is equal to
\begin{align}\label{E:cozm}
&\sum_{i=0}^m\frac{(q^\beta;q)_i}{(q;q)_i}q^{\beta(m-i)}.
\end{align}
To finish the proof, we want to show that (\ref{E:cozm}) is equal to $(q^{\beta+1};q)_m/(q;q)_m$. This can be done by setting $z=q^\beta$ in the following identity
\begin{align}\label{E:anidentity}
\sum_{i=0}^m\frac{(z;q)_i}{(q;q)_i}z^{m-i}=(qz;q)_m/(q;q)_m.
\end{align}
This identity can easily be proved by making induction on $m$ and noticing that
\begin{align}
\sum_{i=0}^{m+1}\frac{(z;q)_i}{(q;q)_i}z^{m+1-i}
=z\sum_{i=0}^m\frac{(z;q)_i}{(q;q)_i}z^{m-i}+\frac{(z;q)_{m+1}}{(q;q)_{m+1}}.
\end{align}
\end{proof}

 Now, let us turn to study the contraction function $F_{\boldsymbol{\beta},q}[s;t]$. Recall that it is defined by the following
\begin{align}
&F_{\boldsymbol{\beta},q}[s;t]=F_{\boldsymbol{\beta},q}(z_1,\dotsc,z_s;w_1,\dotsc,w_t)\\\nonumber
&=\prod_{1\leq i< j\leq
s}(z_i/z_j;q)_{\beta_i}(qz_j/z_i;q)_{\beta_j}
\prod_{i=1}^s\prod_{j=1}^t(z_i/w_j;q)_{\beta_i}^{-1},
\end{align}
where $\boldsymbol{\beta}=(\beta_1,\dotsc,\beta_s)$ is a sequence of positive integers. 
The Laurent polynomial $F_{\boldsymbol{\beta},q}[s;1]$ has an important property,
which is crucial in our proofs later.
\begin{proposition}\label{P:Fbetaqstw1}
 One has the follow \emph{splitting} formula for $F_{\boldsymbol{\beta},q}[s;1]$:
\begin{align}\label{F:split}
\prod_{1\leq i< j\leq s}\Big(\frac{z_i}{z_j};q\Big)_{\beta_i}\Big(\frac{qz_j}{z_i};q\Big)_{\beta_j}
\prod_{i=1}^s\Big(\frac{z_i}{w_1};q\Big)_{\beta_i}^{-1}=\sum_{a=1}^s\sum_{b=0}^{\beta_a-1}G_{a,b}
\end{align}
where
\begin{align}
&G_{a,b}=(1-q^bz_a/w_1)^{-1}B_{a,b} \prod_{1\leq i<j\leq s}^{i,j\neq
a}(z_i/z_j;q)_{\beta_i}(qz_j/z_i;q)_{\beta_j},
\end{align} with
\begin{align}\label{E:Bab}
B_{a,b}=\frac{q^{b\sum_{k=1}^{a-1}\beta_k+(b+1)\sum_{k=a+1}^s\beta_k}}{(q^{-b};q)_b(q;q)_{\beta_a-b-1}}
&\times\prod_{1\leq
i<a} \Big(q^{1-\beta_i}\frac{z_a}{z_i};q\Big)_{b}\Big(q^{b+1}\frac{z_a}{z_i};q\Big)_{\beta_a-b}\\\nonumber
&\times\prod_{a<j\leq s}\Big(q^{-\beta_j}\frac{z_a}{z_j};q\Big)_{b+1}\Big(q^{b+1}\frac{z_a}{z_j};q\Big)_{\beta_a-b-1}.
\end{align}
\end{proposition}

%If the term $z_1^{k_1}\dotsm z_s^{k_s}w_1^{-m_1}\dotsm w_t^{-m_t}$
%appears in the expansion of $F_{\beta,q}[s;t]$, then there is a
%$k_i\geq m_1$.
\begin{proof}

Multiplying both sides of (\ref{F:split}) by $\prod_{1\leq i\leq
s}(z_i/w_1;q)_{\beta_i}$, we want to show that
\begin{align}\label{E:Gab'}
\prod_{1\leq i< j\leq s}\Big(\frac{z_i}{z_j};q\Big)_{\beta_i}\Big(\frac{qz_j}{z_i};q\Big)_{\beta_j}=\sum_{a=1}^s\sum_{b=0}^{\beta_a-1}H_{a,b},
\end{align}
where $H_{a,b}=G_{a,b}\prod_{1\leq i\leq
s}(z_i/w_1;q)_{\beta_i}$ is a polynomial in $w_1^{-1}$ of degree $\beta_1+\dotsb+\beta_s-1$. Look both sides of (\ref{E:Gab'}) as polynomials in $w_1^{-1}$.
We only need to show that each of the $|\boldsymbol{\beta}|=\beta_1+\dotsb+\beta_s$
evaluations $w_1=q^{b}z_a$ ($a=1,\dotsc,s; b=0,\dotsc,\beta_a-1$) to
the right side of (\ref{E:Gab'}) results in the left side of it. Notice that $H_{a,b}$ is the unique
polynomial which does not vanish
when evaluating $w_1=q^{b}z_a$. Hence we only need to show that
\begin{align}
\prod_{1\leq i< j\leq s}\Big(\frac{z_i}{z_j};q\Big)_{\beta_i}\Big(\frac{qz_j}{z_i};q\Big)_{\beta_j}=H_{a,b}|_{w_1=q^{b}z_a}.
\end{align}
 Canceling $\prod_{1\leq
i<j\leq s}^{i,j\neq a}(z_i/z_j;q)_{\beta_i}(qz_j/z_i;q)_{\beta_j}$
from both sides of it, one needs to show:
\begin{align*}
&\prod_{1\leq i<a}\Big(\frac{z_i}{z_a};q\Big)_{\beta_i}\Big(\frac{qz_a}{z_i};q\Big)_{\beta_a}
\prod_{a< j\leq s}\Big(\frac{z_a}{z_j};q\Big)_{\beta_a}\Big(\frac{qz_j}{z_a};q\Big)_{\beta_j}\\
&=\prod_{b\neq k=0}^{\beta_a-1}\Big(1-\frac{q^kz_a}{w_1}\Big)\prod_{a\neq i=1}^s\Big(\frac{z_i}{w_1};q\Big)_{\beta_i}\Big|_{w_1=q^bz_a}B_{a,b}.
\end{align*}
Rewrite this equation such that $B_{a,b}$ is on one side and compare it with (\ref{E:Bab}); we find that we only need to prove the following two identities:
\begin{align}\label{I:cancel1}
&\frac{(qz_a/z_i;q)_{\beta_a}(z_i/z_a;q)_{\beta_i}}{(q^{-b}z_i/z_a;q)_{\beta_i}}
=q^{b\beta_i}\Big(q^{1-\beta_i}\frac{z_a}{z_i};q\Big)_{b}\Big(q^{b+1}\frac{z_a}{z_i};q\Big)_{\beta_a-b},\\
\label{I:cancel2}
&\frac{(z_a/z_j;q)_{\beta_a}(qz_j/z_a;q)_{\beta_j}}{(q^{-b}z_j/z_a;q)_{\beta_j}}
=q^{(b+1)\beta_j}\Big(q^{-\beta_j}\frac{z_a}{z_j};q\Big)_{b+1}\Big(q^{b+1}\frac{z_a}{z_j};q\Big)_{\beta_a-b-1}.
\end{align}

For the identity (\ref{I:cancel1}), we multiply both the denominator and numerator of its left side by
 $(q^{\beta_i-b}z_i/z_a;q)_b$. The denominator turns into $(q^{-b}z_i/z_a;q)_{\beta_i+b}$; the numerator becomes
\begin{align}\label{Comp:1}
&(q^{\beta_i-b}z_i/z_a;q)_b(qz_a/z_i;q)_{\beta_a}(z_i/z_a;q)_{\beta_i}\\\label{Comp:2}
&=(q^{\beta_i-b}z_i/z_a;q)_b(qz_a/z_i;q)_{b}(q^{b+1}z_a/z_i;q)_{\beta_a-b}(z_i/z_a;q)_{\beta_i},
\end{align}
where we split the middle q-shifted factorial in
(\ref{Comp:1}) and obtain (\ref{Comp:2}). Then we apply the following
Lemma \ref{L:interchange} to the product of the first two factorials
in (\ref{Comp:2}); now the new numerator is equal to
\begin{align}\label{Comp:3}
&q^{b\beta_i}(q^{-b}z_i/z_a;q)_b(q^{1-\beta_i}z_a/z_i;q)_b(q^{b+1}z_a/z_i;q)_{\beta_a-b}(z_i/z_a;q)_{\beta_i}.
 \end{align}
Finally, we combine the
first and forth factorials in (\ref{Comp:3}) and it equals:
\begin{align*}
&q^{b\beta_i}(q^{1-\beta_i}z_a/z_i;q)_b(q^{b+1}z_a/z_i;q)_{\beta_a-b}(q^{-b}z_i/z_a;q)_{\beta_i+b}.
\end{align*}
Divide this by the new denominator; we get the right side of (\ref{I:cancel1}). This finishes the proof of identity (\ref{I:cancel1}).

The identity (\ref{I:cancel2}) is proved similarly, but one starts by multiplying $(q^{\beta_j-b}z_j/z_a;q)_{b+1}$ to both the denominator and numerator of its left side. The following steps are done exactly the same way.
%\end{align}
\end{proof}

\begin{lemma}\label{L:interchange}
For a non-negative integer $b$ and integers $k,l$, we have:
\begin{equation*}
(q^kz_i/z_a;q)_b(q^lz_a/z_i;q)_b=q^{b(l+k+b-1)}(q^{1-b-l}z_i/z_a;q)_b(q^{1-b-k}z_a/z_i;q)_b.
\end{equation*}
\end{lemma}
\begin{proof}
We use the following identity
\begin{equation*}(1-q^mz_i/z_a)(1-q^nz_a/z_i)=(q^n-z_i/z_a)(q^m-z_a/z_i)
\end{equation*}
 in the following computations:
\begin{align*}
&(q^kz_i/z_a;q)_b(q^lz_a/z_i;q)_b=\prod_{j=0}^{b-1}(1-q^{k+j}z_i/z_a)(1-q^{l+j}z_a/z_i)\\
&=\prod_{j=0}^{b-1}(q^{l+j}-z_i/z_a)(q^{k+j}-z_a/z_i)\\
&=\prod_{j=0}^{b-1}q^{k+l+2j}(1-q^{-(l+j)}z_i/z_a)(1-q^{-(k+j)}z_a/z_i)\\
&=q^{b(l+k+b-1)}(q^{1-b-l}z_i/z_a;q)_b(q^{1-b-k}z_a/z_i;q)_b.
\end{align*}
\end{proof}
We will apply the splitting property of $F_{\boldsymbol{\beta},q}[s;1]$
(Proposition \ref{P:Fbetaqstw1}) to prove an orthogonality relation and compute some constant terms.
To warm up and for latter application, let us first give a new proof of the q-Dyson constant term conjecture, which was given by Andrews \cite{A}  and has firstly been proved by Zeilberger and Bressoud.
\begin{theorem}\cite{Z}\label{T:qDysonCT} Let $s,\beta_1,\dotsc,\beta_s$ be positive integers.
The constant term of $F_{\boldsymbol{\beta},q}(z_1,\dotsc,z_s)$ is:
\begin{equation}
\operatorname{C.T.} \prod_{1\leq i<j\leq s}\Big(\frac{z_i}{z_j};q\Big)_{\beta_i}\Big(\frac{qz_j}{z_i};q\Big)_{\beta_j}
=\frac{(q;q)_{\beta_1+\beta_2+\dotsb+\beta_s}}{(q;q)_{\beta_1}(q;q)_{\beta_2}\dotsm(q;q)_{\beta_s}}.
\end{equation}
\end{theorem}
\begin{proof} Let us prove it by induction on $s$. It is trivially true if $s=1$.
For $s\geq2$, one observes that $\operatorname{C.T}.F_{\boldsymbol{\beta},q}(z_1,\dotsc,z_s)$ is
the constant term of $F_{\boldsymbol{\beta},q}[s,1]$ (as a Laurent polynomial in $z_i$'s and $w_1$) in Proposition \ref{P:Fbetaqstw1}.
Note that $G_{a,b}$ is a polynomial in $z_a$, and that the constant term $C_a$ of $\prod_{1\leq i<j\leq s}^{i,j\neq a}(z_i/z_j;q)_{\beta_i}(qz_j/z_i)_{\beta_j}$
is known by the induction assumption. We use the notation $|\boldsymbol{\beta}|=\beta_1+\dotsb+\beta_s$ in the following computations:
\begin{align*}
&\operatorname{C.T.}F_{\boldsymbol{\beta},q}(z_1,\dotsc,z_s)\\
&=\sum_{a=1}^s\sum_{b=0}^{\beta_a-1}q^{b\sum_{j=1}^{a-1}\beta_j+(b+1)\sum_{j=a+1}^{s}\beta_j}
(q^{-b};q)_b^{-1}(q;q)_{\beta_a-1-b}^{-1}C_a\\
&=\sum_{a=1}^s C_aq^{\beta_{a+1}+\dotsb+\beta_s}\sum_{b=0}^{\beta_a-1}
q^{b(|\boldsymbol{\beta}|-\beta_a)}(q^{-b};q)_b^{-1}(q;q)_{\beta_a-1-b}^{-1}\\
&=\sum_{a=1}^s C_aq^{\beta_{a+1}+\dotsb+\beta_s}
(q^{\beta_a};q)_{|\boldsymbol{\beta}|-\beta_a}/(q;q)_{|\boldsymbol{\beta}|-\beta_a}\\
%\end{align*}
%\begin{align*}
&=\sum_{a=1}^s \frac{(q;q)_{|\boldsymbol{\beta}|-\beta_a}(q;q)_{\beta_a}}{(q;q)_{\beta_1}\dotsm(q;q)_{\beta_s}}
q^{\beta_{a+1}+\dotsb+\beta_s}
(q^{\beta_a};q)_{|\boldsymbol{\beta}|-\beta_a}/(q;q)_{|\boldsymbol{\beta}|-\beta_a}\\
&=\frac{(q;q)_{|\boldsymbol{\beta}|}}{(q;q)_{\beta_1}\dotsm(q;q)_{\beta_s}(1-q^{|\boldsymbol{\beta}|})}
\sum_{a=1}^s(1-q^{\beta_a}) q^{\beta_{a+1}+\dotsb+\beta_s}\\
&=\frac{(q;q)_{|\boldsymbol{\beta}|}}{(q;q)_{\beta_1}\dotsm(q;q)_{\beta_s}},
\end{align*}
where we use Lemma \ref{C:cozqbeta} for the third equal sign.
\end{proof}
Another direct result of splitting property of $F_{\boldsymbol{\beta},q}[s;1]$ (Proposition \ref{P:Fbetaqstw1}) is the following
q-Dyson orthogonality relation which was  conjectured by Kadell in \cite{K}.
\begin{theorem}\label{C:Kadell}
Let $n$ be a positive integer. In the expansion of
\begin{align}
F_{\boldsymbol{\beta},q}[s;1]=\prod_{1\leq i< j\leq
s}\Big(\frac{z_i}{z_j};q\Big)_{\beta_i}\Big(\frac{qz_j}{z_i};q\Big)_{\beta_j}
\prod_{i=1}^s\Big(\frac{z_i}{w_1};q\Big)_{\beta_i}^{-1},
\end{align}
the term $z_1^{k_1}z_2^{k_2}\dotsm z_s^{k_s}/w_1^n$ does not appear
if $k_i<n$ for all $i=1,2,\dotsc,s$. Moreover, the coefficient
of the term $z_a^n/w_1^n$ in the expansion of
$F_{\boldsymbol{\beta},q}[s;1]$ is:
\begin{align}\label{E:kadelconstant}
%&q^{\sum_{j=a+1}^s\beta_j}\frac{(1-q^{\beta_a})}{(q^{|\boldsymbol{\beta}|-\beta_a+1};q)_n}
%\frac{(q;q)_{|\boldsymbol{\beta}|+\beta_s+n-1}}{(q;q)_{\beta_1}(q;q)_{\beta_2}\dotsm(q;q)_{\beta_s}}\\
%&=
q^{\sum_{j=a+1}^s\beta_j}\frac{(1-q^{\beta_a})(q;q)_{|\boldsymbol{\beta}|-\beta_a}}
{(q;q)_{\beta_1}(q;q)_{\beta_2}\dotsm(q;q)_{\beta_s}}
(q^{|\boldsymbol{\beta}|-\beta_a+n+1};q)_{\beta_a-1},
\end{align}
where $|\boldsymbol{\beta}|=\beta_1+\beta_2+\dotsb+\beta_s$.
\end{theorem}
\begin{proof}
The first statement is clear because $B_{a,b}$ is a polynomial in
$z_a$. One notices that only the $G_{a,b}$'s ($b=0,1,\dotsc,\beta_a$) contribute to the
coefficient of $z_a^n/w_1^n$, and that the constant term of
$\prod_{1\leq i<j\leq s}^{i,j\neq a}(z_i/z_j;q)_{\beta_i}(qz_j/z_i;q)_{\beta_j}$
is
\begin{equation}
C_a=\frac{(q;q)_{\beta_a}(q;q)_{|\boldsymbol{\beta}|-\beta_a}}{(q;q)_{\beta_1}\dotsm(q;q)_{\beta_s}}
\end{equation}
by Theorem \ref{T:qDysonCT}.
Hence the coefficient of $z_a^n/w_1^n$ is:
\begin{align*}
&\sum_{b=0}^{\beta_a-1}q^{nb}q^{b\sum_{j=1}^{a-1}\beta_j+(b+1)\sum_{j=a+1}^{s}\beta_j}
(q^{-b};q)_b^{-1}(q;q)_{\beta_a-1-b}^{-1}C_a\\
&=C_aq^{\beta_{a+1}+\dotsb+\beta_s}\sum_{b=0}^{\beta_a-1}
q^{b(n+|\boldsymbol{\beta}|-\beta_a)}(q^{-b};q)_b^{-1}(q;q)_{\beta_a-1-b}^{-1}\\
&=C_aq^{\beta_{a+1}+\dotsb+\beta_s}
(q^{\beta_a};q)_{n+|\boldsymbol{\beta}|-\beta_a}/(q;q)_{n+|\boldsymbol{\beta}|-\beta_a},
\end{align*}
where the last equal sign is from Lemma \ref{C:cozqbeta}.
 One can rewrite this last expression into the wanted form (\ref{E:kadelconstant}).
\end{proof}
To study the vertex operator realization of Macdonald functions, we need to further investigate the Laurent
polynomial $F_{\boldsymbol{\beta},q}[s,t]$ for general $t$. We observe that $F_{\boldsymbol{\beta},q}[s,1]$ appears as a part in the product expression for $F_{\boldsymbol{\beta},q}[s,t]$, so the splitting property of $F_{\boldsymbol{\beta},q}[s,1]$ will be applied again to study $F_{\boldsymbol{\beta},q}[s,t]$. Before stating the theorem, we need first to generalize the dominance ordering to $\mathbb{Z}^n$.
\begin{definition}\cite{M}\label{D:generalizedDom}
For a vector $a=(a_1,\dotsc,a_n)\in
\mathbb{Z}^n$, define $a^+$ as the vector obtained by rearranging
$a_1,\dotsc,a_n$ in weakly decreasing order. Furthermore, for $a\in \mathbb{Z}^n$ and
$b\in \mathbb{Z}^m$, we define $a\geq b$ if $a_1+\dotsb+a_i\geq
b_1+\dotsb+b_i$ for all $i\geq1$, where we set
$0=c_{l+1}=c_{l+2}=\dotsb$ for $c\in \mathbb{Z}^l$.
\end{definition}

\begin{theorem}\label{T:Fbetaqst}
 For $\textbf{k}=(k_1,\dotsc,k_s)\in \mathbb{Z}^s$,
$\textbf{m}=(m_1,\dotsc,m_t)\in \mathbb{Z}_{\geq0}^t$, if the term
$z^\textbf{k}/w^\textbf{m}=z_1^{k_1}\dotsm z_s^{k_s}w_1^{-m_1}\dotsm
w_t^{-m_t}$ appears in the expansion of $F_{\boldsymbol{\beta},q}[s;t]$ (defined by (\ref{F:gqdlp})), then
$\textbf{k}^+\geq \textbf{m}^+$. Moreover, in the case that all $\beta_i=\beta$, if $s=t=l(\la)$, then the coefficient of $z^\lambda/w^\lambda$
in $F_{\beta,q}[s;t]$ is
\begin{align}\label{F:cla}
c_\la
&=q^{\beta [(l(\la)^2-\sum_im_i(\la)^2]/2}\prod_{i\geq1}
\frac{(q^\beta;q^{\beta})_{m_i(\la)}}{(1-q^\beta)^{m_i(\la)}}\times
\prod_{i=1}^{l(\la)}\frac{(q^\beta;q)_{\la_i+(s-i)\beta}}{(q;q)_{\la_i+(s-i)\beta}}.
\end{align}
\end{theorem}
\begin{remark}
We also have an expression for the coefficient of $z^\la/w^\la$ in $F_{\boldsymbol{\beta},q}[s,t]$.
The expression we found is much more complicated than that for $F_{{\beta},q}[s,t]$. It involves summations and products of sequences.
We thus only give the formula (\ref{F:cla}) for the simplicity of it and also that it is enough for what follows.
\end{remark}
\begin{proof} Observe that $F_{\boldsymbol{\beta},q}[s,t]$ can be written as the product of $F_{\boldsymbol{\beta},q}[s,1]$ and
 $\prod_{i=1}^s\prod_{j=2}^t(z_i/w_j;q)_{\beta_i}^{-1}$. The splitting formula for $F_{\boldsymbol{\beta},q}[s,1]$ (Proposition \ref{P:Fbetaqstw1}) enables us to write
\begin{equation}\label{F:splitFbetaqst}
F_{\boldsymbol{\beta},q}[s;t]=\sum_{a=1}^s\sum_{b=0}^{\beta_a-1}A_{a,b}B_{a,b}C_aD_a,
\end{equation} with $B_{a,b}$ given by (\ref{E:Bab})
and
\begin{align*}
&A_{a,b}=(1-q^bz_a/w_1)^{-1},\\
&C_a=F_{\boldsymbol{\beta},q}(z_1,\dotsc,\widehat{z_a},\dotsc,z_s;w_2,\dotsc,w_t),\\
&D_a=\prod_{j=2}^t(z_a/w_j;q)_{\beta_a}^{-1},
\end{align*}
where-and in the following-the hat sign $\widehat{}$ means omission.

For the first statement, we only need to
prove the case that when $\textbf{m}$ is a partition (thus
$\textbf{m}^+=\textbf{m}$), because $F_{\boldsymbol{\beta},q}[s;t]$ is symmetric in $w_1,\dotsc,w_t$.
We make induction on $t$. It is true when
$t=1$ by Proposition \ref{P:Fbetaqstw1}. For general $t$,
if a term $z^{\textbf{k}}/w^{\textbf{m}}$ appears, there should be
four terms from $A_{a,b},B_{a,b},C_a,D_a$ respectively such that
their product is $z^\textbf{k}/w^\textbf{m}$. Let us say the product
is of the following form:
\begin{align}\label{E:4termsproduct}
\frac{z^\textbf{k}}{w^\textbf{m}}=\frac{z_a^m}{w_1^m}\cdot
\frac{z_a^{b_1+\dotsb+\widehat{b_a}+\dotsb+b_s}}{z_1^{b_1}\dotsb\widehat{z_a^{b_a}}\dotsm
z_s^{b_s}}\cdot \frac{z_1^{c_1}\dotsm\widehat{z_a^{c_a}}\dotsm
z_s^{c_s}}{w_2^{c'_2}\dotsm w_t^{c'_t}} \cdot \frac
{z_a^{d_2\dotsb+d_t}}{w_2^{d_2}\dotsm w_t^{d_t}},
\end{align}
where all those powers-$m$, $b_i$'s, $c_i$'s, $c'_i$'s and $d_i$'s-are non-negative integers. The third term in (\ref{E:4termsproduct})
is from $C_a$. By the induction assumption one has:
\begin{equation} \label{Ine:cc'}
(c_{i_1},\dotsc,\widehat{c_{i_a}},\dotsc,c_{i_s})\geq(c'_{2},\dotsc,c'_{t}),
\end{equation}
where $(i_1,\dotsc,\widehat{i_a},\dotsc,i_s)$ is a rearrangement of $1,\dotsc,\widehat{a},\dotsc,s$ such that
the vector $(c_{i_1},\dotsc,\widehat{c_{i_a}},\dotsc,c_{i_s})$ is in decreasing order,
and we use that $(c'_2,\dotsc,c'_t)^+$ is greater than or equal to $(c'_{2},\dotsc,c'_{t})$ by Definition \ref{D:generalizedDom}.
Now we obtain the following three relations:
\begin{align*}
&(0,c_{i_1},\dotsc,\widehat{c_{i_a}},\dotsc,c_{i_s})\geq(0,c'_{2},\dotsc,c'_{t}),\\
&(m+\sum_{i_a\neq
j=1}^sb_j,-b_{i_1},\dotsc,\widehat{-b_{i_a}},\dotsc,-b_{i_s})\geq
(m,0,\dotsc,0),\\
&(d_2\dotsb+d_t,0,\dotsc,0)\geq(0,d_{2},\dotsc,d_{t}).
\end{align*}
The first relation is from (\ref{Ine:cc'}) and the other two come
from the fact that $b_i$'s and $d_j$'s are
 non-negative. Sum the three inequalities up; we find
 \begin{align}
 (k_a,k_{i_1},\dotsc,\widehat{k_{i_a}},\dotsc,k_{i_s})\geq
 (m_1,m_{2},\dotsc,m_{t}),
 \end{align}
which $\textbf{k}^+>\textbf{m}^+$ follows.

We now turn to the second statement. We also prove it by induction on $l(\la)$.
First, it is true if $l(\la)=1$, by Lemma \ref{C:cozqbeta} . We assume that the theorem is true when $l(\la)=s-1$. For the case that $l(\la)=s$, let us first write
$\lambda=(a_1^{n_1}a_2^{n_2}\cdots)$; i.e.,the parts of $\la$ are $n_1$ $a_1$'s ($a_1=\lambda_1$), $n_2$ $a_2$'s, etc., with $a_1>a_2>\dotsb$. Then we have
\begin{equation}
z^\la/w^\la=\frac{z_1^{a_1}z_2^{a_1}\cdots z_{n_1}^{a_1}z_{n_1+1}^{a_2}z_{n_1+2}^{a_2}\cdots z_{n_1+n_2}^{a_2}\cdots}
{w_1^{a_1}w_2^{a_1}\cdots w_{n_1}^{a_1}w_{n_1+1}^{a_2}w_{n_1+2}^{a_2}\cdots w_{n_1+n_2}^{a_2}\cdots}.
\end{equation}
Again, we look at the splitting formula (\ref{F:splitFbetaqst}), but with all $\beta_i=\beta$. To obtain $z^\la/w^\la$, one first takes $(q^bz_a/w_1)^{a_1}$
from $A_{a,b}$ here $a$ is from $\{1,2,\dotsc,n_1\}$. Once this term is taken, one can only take the constant term
which equals $q^{b(a-1)\beta+(b+1)(s-a)\beta}/[(q^{-b};q)_b(q;q)_{\beta-b-1}]$ from $B_{a,b}$ and $1$ from $D_a$; otherwise, the power of $z_a$ will be greater than $a_1$. The other powers of $z_i$'s ($i\neq a$) and $w_j$'s ($j\neq1$) are from $C_a$. Thus we get the following iterative formula:
\begin{align}\label{F:induction4clambda}
c_\la=&\sum_{a=1}^{n_1}\sum_{b=0}^{\beta-1}q^{b\la_1}\frac{q^{b(a-1)\beta+(b+1)(s-a)\beta}}{(q^{-b};q)_b(q;q)_{\beta-b-1}}c_{\la^{-}},
\end{align}
where $\la^-=(\la_2,\la_3,\dotsc,\la_s)=(a_1^{n_1-1}a_2^{n_2}a_3^{n_3}\cdots)$, and $c_{\la^-}$ is the coefficient of the following term in $F_{\beta,q}(z_1,\dotsc,\widehat{z_a},\dotsc,z_s;w_2,\dotsc,w_t)$
$$\frac{z_1^{a_1}z_2^{a_1}\cdots \widehat{z_{a}^{a_1}}\cdots z_{n_1}^{a_1}z_{n_1+1}^{a_2}z_{n_1+2}^{a_2}\cdots z_{n_1+n_2}^{a_2}\cdots}
{w_2^{a_1}\cdots w_{n_1}^{a_1}w_{n_1+1}^{a_2}w_{n_1+2}^{a_2}\cdots w_{n_1+n_2}^{a_2}\cdots}.$$
Note that this coefficient $c_{\la^-}$ remains to be the same as $a$ ranges in the set $\{1,2,\dotsc,n_1\}$, and it is obviously independent of $b$.
 Thus one can factor it out of the double summation in (\ref{F:induction4clambda}). Then one uses Lemma \ref{C:cozqbeta} to find
\begin{align}
c_\la%&=c_{\la^{-}}\sum_{a=1}^{n_1}\sum_{b=0}^{\beta-1}q^{b\la_1}\frac{q^{b(a-1)\beta+(b+1)(s+1-a)\beta}}{(q^{-b};q)_b(q;q)_{\beta-b-1}}\\
%     &=c_{\la^{-}}\sum_{a=1}^{n_1}q^{(s+1-a)\beta}\frac{(q^\beta;q)_{\la_1+s\beta}}{(q;q)_{\la_1+s\beta}}\\
%     &=c_{\la^{-}}\frac{(q^\beta;q)_{\la_1+s\beta}}{(q;q)_{\la_1+s\beta}}\sum_{a=1}^{n_1}q^{(s+1-a)\beta}\\
     &=c_{\la^{-}}\frac{(q^\beta;q)_{\la_1+(s-1)\beta}}{(q;q)_{\la_1+(s-1)\beta}}q^{(s-n_1)\beta}\frac{1-q^{n_1\beta}}{1-q^\beta}.
\end{align}
The $c_{\la^{-}}$ has the following expression by the induction assumption:
\begin{align*}
c_{\la^{-}}=&q^{\beta ((s-1)^2-\sum_im_i(\la)^2+n_1^2-(n_1-1)^2)/2}\frac{(q^\beta;q^{\beta})_{n_1-1}}{(1-q^\beta)^{n_1-1}}\\
&\times\prod_{\la_1\neq i\geq1}\frac{(q^\beta;q^{\beta})_{m_i(\la)}}{(1-q^\beta)^{m_i(\la)}}
\prod_{i=2}^{s}\frac{(q^\beta;q)_{\la_i+(s-i)\beta}}{(q;q)_{\la_i+(s-i)\beta}}.
\end{align*}
Note that $l(\la)=l(\la^{-})+1=s, n_1=m_{\la_1}(\la)$. We then write $c_\la$ into the wanted form (\ref{F:cla}).
\end{proof}

\section{Almost-rectangular Macdonald functions and Jack functions}
Now we return to the study of Macdonald functions of almost rectangular shapes using the vertex operators and the orthogonality relation in the previous sections. Recall that a partition $\la$ is of almost rectangular shape if $\la$ equals $((k+1)^t,k^s)$ for a pair of positive integers $k,s$ and a non-negative integer $t$, and this $\la$ is of rectangular shape when $t=0$. These partitions have the following property on dominance order.
\begin{lemma}\label{L:almostrec} If $\mu$ is of the same weight as
the almost rectangular partition $\rho=((k+1)^t,k^s)$, then
$\mu\geq\rho$ if and only if $l(\mu)\leq s+t$.
\end{lemma}
This property can be seen from the Young diagram of $\rho$.
As to the formal proof, the ``only if'' part is simple by definition and the ``if'' part is Lemma 3.7 in \cite{CJ1}.

Now we are ready to give the vertex operator realization of
Macdonald functions of almost rectangular shapes. Let us first recall that we define the operators $X_i$'s on $V$ at the beginning of Section 2 and the lowering operator $D_i$'s in Definition \ref{D:Di}. Let $\la=(\la_1,\la_2,\dotsc,\la_n)$ be a partition of length $n$; we use notation $X_{-\la}$ for the operator product $X_{-\la_1}X_{-\la_s}\dotsm X_{-\la_n}$, and also recall that $g_\la=Q_{\la_1}Q_{\la_2}\cdots$.
\begin{theorem}\label{T:main}
Let an almost rectangular partition $\rho$ be $((k+1)^t,k^s)$, with $k,s>0$
and $t\geq0$, and $\beta$ be a positive integer. We have
\begin{align}\label{F:VerRecMac}
&X_{-\rho}.1\otimes
e^{-(s+t)\eta/2}=\epsilon_{q}(\beta,s+t)C_\rho Q_{\rho}(q,q^\beta)\otimes e^{(s+t)\eta/2},\\
\label{F:qDysonMac}
&C_\rho Q_{\rho}(q,q^\beta)=\prod_{1\leq i<j\leq s+t}\Big(\frac{D_i}{D_j};q\Big)_\beta
\Big(\frac{qD_j}{D_i};q\Big)_\beta.g_\rho(q,q^\beta),
%\big(Q_{k+1}(q,q^\beta)\big)^{t}\big(Q_{k}(q,q^\beta)\big)^{s},
\end{align}
where \begin{align*}
&\epsilon_q(\beta,n)=(-1)^{\beta
n(n-1)/2}q^{-\beta(\beta+1)n(n-1)/4},\\
&C_\rho =q^{st\beta}
\frac{(q;q^\beta)_s(q;q^\beta)_t(q;q)_{(s+t)\beta}}{(q;q^\beta)_{s+t}(q;q)_\beta^{s+t}}.
\end{align*}
\end{theorem}
\begin{proof}
Up to the multiplication of $e^{nh}$, $X_{-\rho}.1\otimes
e^{-(s+t)h/2}$ is a linear combination of $g_\mu(q,q^\beta)$'s with $\mu\geq\rho$ by formula (\ref{F:Xlambdaraising}) and Lemma \ref{L:almostrec}.
Furthermore, its scalar product with $g_\mu(q,q^\beta)$ for those $\mu>\rho$ is zero by formula (\ref{F:XYC.T.}) and Theorem \ref{T:Fbetaqst}.
By Lemma \ref{D:Qlambda}, we know that (\ref{F:VerRecMac}) is true for some constant $C_\rho$. Then (\ref{F:qDysonMac}) follows by formula (\ref{F:Xlambdaraising}).
%the two expressions follow by combining Lemmas \ref{L:Xlambda}, \ref{L:almostrec} and Theorem \ref{T:Fbetaqst}.
What remains to be done is to fix the constant $C_\rho$.
Applying Lemma \ref{L:Macnorm} to $\rho=((k+1)^t,k^s)$, one has
\begin{align}\label{F:rectnorm}
&\langle Q_{\rho}(q,q^\beta),Q_{\rho}(q,q^\beta)\rangle=\\\nonumber
&\frac{(q^{s\beta};q)_k(q^\beta;q^\beta)_{s-1}}{(q;q)_{k-1}(q^k;q^\beta)_s}
\frac{(q^{1+(s+t)\beta};q)_k(q^{1+(s+1)\beta};q^\beta)_{t-1}}
{(q^{2+s\beta};q)_{k-1}(q^{k+1+s\beta};q^\beta)_t}
\frac{(q^\beta;q^\beta)_t}{(q;q^\beta)_t}.
\end{align}
Combining formula (\ref{F:XYC.T.}) and formula (\ref{F:cla}) in Theorem \ref{T:Fbetaqst}, we find
\begin{align*}
&\langle (X_{-\rho}.1\otimes e^{-(s+t)h/2}, g_{\rho}\otimes e^{(s+t)h/2}\rangle\\
&=\epsilon_q(\beta,s+t)%q^{-\beta(s+t)}[s]_{q^{-\beta}}![t]_{q^{-\beta}}!
%\frac{K_{\beta,q}(k+1,s+t)}{K_{\beta,q}(k+1,s)}
%K_{\beta,q}(k,s)
q^{st\beta}\frac{(q^\beta;q^\beta)_s(q^\beta;q^\beta)_t}{(1-q^\beta)^{s+t}}
\prod_{i=s}^{s+t-1}\frac{(q^\beta;q)_{k+1+i\beta}}{(q;q)_{k+1+i\beta}}\prod_{i=0}^{s-1}\frac{(q^\beta;q)_{k+i\beta}}{(q;q)_{k+i\beta}}\\
&=\epsilon_q(\beta,s+t)\frac{q^{st\beta}(q^{2\beta};q^\beta)_{s-1}(q^{2\beta};q^\beta)_{t-1}(q^\beta;q)_{k+1+(s+t-1)\beta}(q^\beta;q)_{k+(s-1)\beta}}
{(q;q)_\beta^{s+t-2}(q;q)_{k+1+s\beta}(q;q)_k(q^{k+1+(s+1)\beta};q^\beta)_{t-1}(q^{k+\beta};q^\beta)_{s-1}}.
\end{align*}
%Set $X_{-\rho}.e^{-l(\rho)h/2}=\epsilon_q(\beta,s+t)CQ_\rho$, then
Notice that one has
\begin{align*}
\langle (X_{-\rho}.1\otimes e^{-(s+t)h/2}, g_{\rho}\otimes e^{(s+t)h/2}\rangle
=\epsilon_q(\beta,s+t)C_\rho\langle Q_\rho,Q_\rho\rangle,
\end{align*}
by Lemma \ref{D:Qlambda}. Combining these formulae, one finds an expression for $C_\rho$ which can be written into the wanted form.
\end{proof}

\subsection{Jack function of almost rectangular shapes} Let us study the special case of Jack functions.
We will show that our theorem on Macdonald functions implies the hyperdeterminant
formula or Jacobi-Trudi formula for Jack functions of almost rectangular shapes \cite{BBL}.
Recall that the Macdonald function $Q_\la(q,q^\beta)$ goes to Jack function $Q_\la(\beta^{-1})$ when $q$ goes to $1$. Taking the limit $q\rightarrow 1$ in Theorem \ref{T:main},
we get the vertex operator
realization of Jack function $Q_\rho(\beta^{-1})$ of almost rectangular shapes and also the following lowering operator formula (\ref{F:DysonAlmostRecJac}) for them.
This specification covers those corresponding results in \cite{CJ1} and \cite{CJ3}, where the arguments
work for $\rho=((k+1)^t,k^s)$ with $s$ restricted to $0$ and $1$. Moreover, one can see in the following (from the proof) that for Jack function of almost rectangular shapes,
 the lowering operator (\ref{F:DysonAlmostRecJac}) and the hyperdeterminant formula (\ref{F:hyperdeterminant}) are equivalent.
 However, we do not know the q-analog of (\ref{F:hyperdeterminant}), i.e., a formula for almost rectangular Macdonald functions which is of the form as (\ref{F:hyperdeterminant})
  and is equivalent to the lowering operator formula (\ref{F:qDysonMac}).

\begin{corollary}\cite{BBL}\label{C:AlReJack}
For an almost rectangular partition $\rho=((k+1)^t,k^s)$ and a positive integer $\beta$,
one has the following hyperdeterminant formula for the Jack function
$Q_{\rho}(\beta^{-1})$:
\begin{align}\label{F:hyperdeterminant}
&\frac{(\beta^{-1})_s(\beta^{-1})_t}{(\beta^{-1})_{s+t}}\frac{((s+t)\beta)!}{(\beta!)^{s+t}}Q_{\rho}(\beta^{-1})=\\ \nonumber
&\sum \prod_{i=1}^{2\beta}\operatorname{sgn}(\sigma_i)\prod_{i=1}^t
Q_{k+1+\sum_{j=1}^\beta\sigma_{\beta+j}(i)-\sigma_j(i)}(\beta^{-1})\prod_{i=1}^s
Q_{k+\sum_{j=1}^\beta\sigma_{\beta+j}(i)-\sigma_j(i)}(\beta^{-1}),
\end{align}
where the sum runs over
$(\sigma_1,\sigma_2\cdots,\sigma_{2\beta})\in
(\mathfrak{S}_{s+t})^{2\beta}$ with $\mathfrak{S}_{n}$ being the symmetric group of degree n, and the Pochhammer symbol
$(x)_n$ is defined to be $x(x+1)\cdots(x+n-1)$.

\end{corollary}

\begin{remark} The rectangular case-the case that $t=0$-was given in \cite{Ma}.
The case that $\beta=1$ gives the Jacob-Trudi formula for almost rectangular Schur functions.
\end{remark}

\begin{proof}
Take the limit $q\rightarrow 1$ in (\ref{F:qDysonMac}). One has
\begin{align}\label{F:DysonAlmostRecJac}
&\frac{(\beta^{-1})_s(\beta^{-1})_t}{(\beta^{-1})_{s+t}}\frac{((s+t)\beta)!}{(\beta!)^{s+t}}Q_{((k+1)^t,k^s)}(\beta^{-1})\\\nonumber
&=\prod_{1\leq i<j\leq s+t}\Big(1-\frac{D_i}{D_j})^\beta
\Big(1-\frac{D_j}{D_i}\Big)^\beta.
\big(Q_{k+1}(\beta^{-1})\big)^{t}\big(Q_{k}(\beta^{-1})\big)^{s}.
\end{align}
Observe that $(1-D_i/D_j)(1-D_j/D_i)=(D_j-D_i)(D_j^{-1}-D_i^{-1})$. We
can re-write the operator in the right side of (\ref{F:DysonAlmostRecJac}) as a
product of Vandermonde determinants. We then expand these determinants into
an alternating sum of monomials. Hence we have the following:
\begin{align*}
&\prod_{1\leq i<j\leq s+t}\Big(1-\frac{D_i}{D_j})^\beta
\Big(1-\frac{D_j}{D_i}\Big)^\beta\\
&=\Big(\sum_{\sigma\in\mathfrak{S}_{s+t}}\operatorname{sgn}(\sigma)\prod_{i=1}^{s+t}D_i^{\sigma(i)-1}\Big)^\beta
\Big(\sum_{\sigma\in\mathfrak{S}_{s+t}}\operatorname{sgn}(\sigma)\prod_{i=1}^{s+t}D_i^{-\sigma(i)+1}\Big)^\beta\\
&=\sum \prod_{i=1}^{2\beta}\operatorname{sgn}(\sigma_i)
\prod_{i=1}^{s+t}D_i^{\sum_{j=1}^\beta\sigma_j(i)-\sigma_{j+\beta}(i)},
\end{align*}
where the summation runs over all $(\sigma_1,\sigma_2,\dotsc,\sigma_{2\beta})\in \mathfrak{S}_{s+t}^{2\beta}$.
Now (\ref{F:hyperdeterminant}) follows by Definition \ref{D:Di} of the lowering operator $D_i$.
\end{proof}

\section{A Frobenious formula for Macdonald functions}
In this section, we first use the technique of vertex operator to find a combinatorial formula for rectangular Macdonald functions.
then we iteratively construct all Macdonald functions with those rectangular ones.
As a result, we find a combinatorial formula expressing Macdonald functions as a linear combination of power sums.
We also give an example to show that how can one express a Macdonald functions as the coefficient of a monomial in a Laurent polynomial.

For convenience, let us write our vertex operator as the following form:
\begin{align}
X(z)=\exp\Big(\sum_{n\geq1}\frac{h_{-n}z^n}{n\epsilon_n}\Big)
M(z)\exp\Big(\sum_{n\geq1}\frac{h_{n}z^{-n}\tau_n}{-n}\Big)
=\sum_nX_{-n}z^n,
\end{align}
where
\begin{align}\label{N:epsilonntaun}
\epsilon_n=\frac{1-q^n}{1-q^{n\beta}} \; \text{ ~~and~~} \;\tau_n=\frac{q^{-n\beta}-q^{n\beta}}{1-q^n}.
\end{align}
Now we define the following notations for a partition $\la$:
\begin{align}\label{N:epsilonlataula}
\epsilon_\la=\epsilon_{\la_1}\epsilon_{\la_2}\cdots \; \text{~~and~~~}\;
  \tau_\la=\tau_{\la_1}\tau_{\la_2}\cdots,
\end{align}
where we set $\epsilon_0=\tau_0=1$. The following definitions will also be used in the combinatorial formula that we are going to give.
\begin{definition}\label{D:partitionnotations}
Let $\la,\mu$ be two partitions, if $m_i(\la)\geq m_i(\mu)$ for all $i$, we %say that $\mu$ is contained in $\la$ in term of multiplicity, and
write $\mu\subset'\la$. In this case, we define the set difference $\la\backslash\mu$ by
 $m_i(\la\backslash\mu)=m_i(\la)-m_i(\mu)$, and the multiplicity binomial
 $\binom{m(\la)}{m(\mu)}$ is defined to be $\prod_i\binom{m_i(\la)}{m_i(\mu)}$.
\end{definition}
Using these notations, we first have the following formulae by computations:
\begin{align}\label{F:generating}
&\exp\Big(\sum_{n\geq1}\frac{h_{-n}z^n}{n\epsilon_n}\big).v\otimes e^{m\eta}
=\sum_{\nu\in\mathcal{P}}z_\nu^{-1}\epsilon_\nu^{-1}p_{\nu}z^{|\nu|}v\otimes e^{m\eta},\\\label{F:Annihilating}
&\exp\Big(\sum_{n\geq1}\frac{h_{n}z^{-n}\tau_n}{-n}\big).p_{\lambda}\otimes e^{m\eta}
=\sum_{\mu\subset'\lambda}\binom{m(\lambda)}{m(\mu)}(-1)^{l(\mu)}\tau_\mu\epsilon_\mu
z^{-|\mu|}p_{\lambda\backslash\mu}\otimes e^{m\eta}.
\end{align}
Then we use these formulae to compute the successive actions of the vertex operator. We found a combinatorial formula as in the following.
\begin{proposition}\label{P:XlaAsPowersum}
For a partition $\la=(\la_1,\dotsc,\la_s)$ of length $s$, a positive integer $\beta$ and $n\in\frac12\mathbb{Z}$, we have
\begin{align}\label{F:CombXlambda}
&X_{-\la_s}\cdots X_{-\la_1}.1\otimes e^{n\eta}\\\nonumber
&=\sum_{\underline{\mu},\underline{\nu}}\prod_{i=1}^s\frac{(-1)^{l(\nu^i)}\tau_{\nu^i}}{z_{\nu^i}}\binom{m(\mu^{i-1})}{m(\mu^i\backslash\nu^i)}
\frac{(-1)^{l(\mu^s)}p_{\mu^s}}{\tau_{\mu^s}\epsilon_{\mu^s}}\otimes e^{(n+s)\eta},
\end{align}
where the sum is over pairs of partition sequences $\underline{\mu}=(\mu^0,\mu^1,\dotsc,\mu^s), \underline{\nu}=(\nu^1,\nu^2,\dotsc,\nu^s)$, such that
$|\mu^i|=\la_1+\dotsb+\la_i-i(i+2n)\beta$ (thus $\mu^0=(0)$), $\nu^i\subset'\mu^i$, and $\mu^i\backslash\nu^i\subset'\mu^{i-1}$.
\end{proposition}
\begin{proof}
By (\ref{F:generating}), it is easy to see that it is true when $s=1$.
 Assume that (\ref{F:CombXlambda}) is true;
 let us consider $X_{-\la_{s+1}}X_{-\la_s}\cdots X_{-\la_1}.1\otimes e^{n\eta}$. To compute it, we act $X(z)$ to (\ref{F:CombXlambda}) and take the coefficient of $z^{\la_{s+1}}$.
 First, act the annihilation part of $X(z)$ to the right side of (\ref{F:CombXlambda}) using formula (\ref{F:Annihilating}); the result is:
\begin{align}\label{E:Annihilating0}
&\sum_{\mu\subset'\mu^s}\sum_{\underline{\mu},\underline{\nu}}
\prod_{i=1}^s\frac{(-1)^{l(\nu^i)}\tau_{\nu^i}}{z_{\nu^i}}\binom{m(\mu^{i-1})}{m(\mu^i\backslash\nu^i)}\\\nonumber
&\qquad\times\frac{(-1)^{l(\mu^s)}}{\tau_{\mu^s}\epsilon_{\mu^s}}\binom{m(\mu^s)}{m(\mu)}
(-1)^{l(\mu)}\tau_\mu\epsilon_\mu z^{-|\mu|}p_{\mu^s\backslash\mu}\otimes e^{(n+s)\eta}.
\end{align}
Acting the middle term to it,
the result is the same as one multiplies $z^{(2(n+s)+1)\beta}$ to (\ref{E:Annihilating0})
and changes $e^{(n+s)h}$ to $e^{(n+s+1)h}$. Then one apply the creation part of $X(z)$ to this result using (\ref{F:generating}), and take the
 coefficient of $z^{\la_{s+1}}$. We find that
$X_{-\la_{s+1}}X_{-\la_s}\cdots X_{-\la_1}.1\otimes e^{n\eta}$ equals the following:
\begin{align}\label{E:Annihilating}
&\sum_{\nu^{s+1}}\sum_{\mu\subset'\mu^s}\sum_{\underline{\mu},\underline{\nu}}
\prod_{i=1}^s\frac{(-1)^{l(\nu^i)}\tau_{\nu^i}}{z_{\nu^i}}\binom{m(\mu^{i-1})}{m(\mu^i\backslash\nu^i)}\\\nonumber
&\times\frac{(-1)^{l(\mu^s)}}{\tau_{\mu^s}\epsilon_{\mu^s}}\binom{m(\mu^s)}{m(\mu)}
(-1)^{l(\mu)}\tau_\mu\epsilon_\mu z_{\nu^{s+1}}^{-1}\epsilon_{\nu^{s+1}}^{-1}p_{(\mu^s\backslash\mu)\cup\nu^{s+1}}\otimes e^{(n+s+1)\eta}.
\end{align}
where $\nu^{s+1}$ subjects to $|\nu^{s+1}|-|\mu|+(2(n+s)+1)\beta=\la_{s+1}$. Finally, we replace the variable $\mu$ with $\mu=\mu^{s}\backslash(\mu^{s+1}\backslash\nu^{s+1})$. 
 Note that $\binom{m(\mu^s)}{m(\mu)}=\binom{m(\mu^s)}{m(\mu^s\backslash\mu)}$; we can write (\ref{E:Annihilating}) into the wanted form.
\end{proof}

From now on, we consider symmetric functions in $\Lambda_{F}$ (not the $V$ as in the previous sections).
\begin{definition}\label{D:gksmu}
For a rectangular partition $R=(k^s)$, a positive integer $\beta$ and a partition $\mu\vdash ks$, we define
\begin{align}
g_{R,\mu}(\beta,q)&=(-1)^{\beta
s(s-1)/2}q^{\beta(\beta+1)s(s-1)/4}\frac{(q;q)_\beta^{s}}{(q;q)_{s\beta}}\frac{(-1)^{l(\mu)}}{\tau_{\mu}\epsilon_{\mu}}\\\nonumber
&\qquad\times\sum_{\underline{\mu},\underline{\nu}}
\prod_{i=1}^s\frac{(-1)^{l(\nu^i)}\tau_{\nu^i}}{z_{\nu^i}}\binom{m(\mu^{i-1})}{m(\mu^i\backslash\nu^i)},
\end{align}
where the sum is over pairs of partition sequences $\underline{\mu}=(\mu^0,\mu^1,\dotsc,\mu^s),\underline{\nu}=(\nu^1,\nu^2,\dotsc,\nu^s)$, such that $\mu^s=\mu$,
$|\mu^i|=ik+i(s-i)\beta$ (thus $\mu^0=(0)$), $\nu^i\subset'\mu^i$, and $\mu^i\backslash\nu^i\subset'\mu^{i-1}$, the notations $\tau_\la, \epsilon_\la$ are defined in (\ref{N:epsilonntaun}) and (\ref{N:epsilonlataula}), and the other notations can be found in Definition \ref{D:partitionnotations}.
\end{definition}
Now let us consider $\la=(k^s)$ in Proposition \ref{P:XlaAsPowersum} and $t=0$ in Theorem \ref{F:VerRecMac}.
 We have the following combinatorial formula for Macdonald function of rectangular shapes.
\begin{corollary}\label{F:RecMacAsPowerSum}
For a rectangular partition $R=(k^s)$ and a positive integer $\beta$, we have
\begin{align}\label{F:comb4Q}
Q_{R}(q,q^\beta)=\sum_{\mu\vdash ks}g_{R,\mu}(\beta,q)p_\mu.
\end{align}
\end{corollary}
\subsection{Iterative construction with rectangular Macdonald functions}
As it was done to Jack function in \cite{CJ3} (see Remark 4.23 there), we can generalize formula (\ref{F:comb4Q}) to Macdonald functions of general shapes. Let us first introduce the definition of the complement of a partition in a rectangular partition.
\begin{definition}\cite{CJ3}
For a rectangular partition $R=(k^s)$, and a partition $\la$ satisfying $\la_1\leq k$ and $l(\la)\leq s$, we define
the complement $\overline{\la}=R-'\la$ of $\la$ in $R$ by
\begin{equation}
\overline{\la}_i=k-\la_{s+1-i} \text{ for } i=1,2,\dotsc,s.
\end{equation}
(Recall that $\mu_i=0$ for all $i> l(\mu)$.)
Specifically, we define $\mathfrak{C}(\la)=(\la_1^{l(\la)})-'\la$ and call it the exact complement of $\la$.
\end{definition}
The exact complement of $\la$ is \emph{simpler} in shape than $\la$ in terms of the number of (lower-right) corners of $\la$. To explain it explicitly, we first give the following:
\begin{definition}\cite{CJ3}
For a partition $\la=(a_1^{n_1}a_2^{n_2}\cdots a_r^{n_r})$ with $a_1>a_2\dotsb>a_r>0$ and $n_1,n_2,\dotsc,n_s>0$, we call $r$ the (lower-right) corner number of $\la$, and denote it as $n_c(\la)=r$.
\end{definition}  For this $\la$ we have $\mathfrak{C}(\la)=(a_1-a_r)^{n_r}(a_1-a_{r-1})^{n_{r-1}}\cdots(a_1-a_2)^{n_2}$, thus $n_c(\mathfrak{C}(\la))=n_c(\la)-1$ if $n_c(\la)>1$. Note that $n_c(\la)=1$ is to say $\la$ is rectangular. Let us define $n_c(\la)=0$ for $\la=(0)$. We see that the map $\mathfrak{C}$ lowers the corner number of every nonzero partition-a partition which is not $(0)$-by one.

What really makes the exact complement interesting is the following property on Macdonald functions.
This property is known in \cite{CJ3} (see Remark 4.9 there).
\begin{lemma}\label{L:ReComplement}
For a rectangular partition $R=(k^s)$, a partition $\mu$ with $\mu_1\leq k$ and $l(\mu)\leq s$ and a partition $\nu$ one has:\\
\begin{equation}\label{E:MacRc}
\langle Q_\mu Q_\nu, Q_R\rangle\neq0 \text{  if and only if } \nu=R-'\mu.
\end{equation}
\end{lemma}

Recall that for a symmetric function $F\in\Lambda_F$, the operator $F^*$ is defined by
$\langle F^*.u,v\rangle=\langle u,Fv\rangle$. Thus the equivalence relation (\ref{E:MacRc}) is to say that $Q_\mu^*.Q_R\doteq Q_{R-'\mu}$, where-and in the following-the sign $\doteq$ means that the two sides are different only by a nonzero scalar multiplication. Specifically,
\begin{equation}\label{E:iterativeengine}
Q_{\mathfrak{C}(\la)}^*.Q_{(\la_1^{l(\la)})}\doteq Q_{\la},
\end{equation}
as $(\la_1^{l(\la)})-'\mathfrak{C}(\la)=\la$.
This property  provides us an iterative method to construct Macdonald functions with those of rectangular shapes. %We mean we want to express a Macdonald function with a formula with only rectangular Macdonald function and the conjugate operator.
For example, let us show that
\begin{equation}\label{E:example}
Q_{(6,6,3,2,2)}\doteq (Q_{(1^1)}^*.Q_{(4^3)})^*.Q_{(6^5)}.
\end{equation}
First, we set $\la=(6,6,3,2,2)$; then $(\la_1^{l(\la)})$ equals $(6^5)$. Therefore we have
\begin{equation} \label{E:exampleQla1}
Q_\la \doteq Q_{\mathfrak{C}(\la)}^*.Q_{(6^5)}
\end{equation}
 by (\ref{E:iterativeengine}). We find that $\mathfrak{C}(\la)$ equals $(4,4,3)$ by definition.  Write $\mu=\mathfrak{C}(\la)$; then $(\mu_1^{l(\mu)})$ equals $(4^3)$. Again we have
  \begin{equation}\label{E:exampleQla2}
  Q_{\mathfrak{C}(\la)}=Q_\mu \doteq Q_{\mathfrak{C}(\mu)}^*.Q_{(4^3)}
  \end{equation}
   But $\mathfrak{C}(\mu)$ equals $(1^1)$. Replacing (\ref{E:exampleQla2}) into (\ref{E:exampleQla1}), we get (\ref{E:example}).

Based on this example, we define a sequence of rectangles associated to a partition as the following.
\begin{definition}
For a nonzero partition $\la$, we define the rectangular filtration of $\la$ $R(\la)=(R_1,R_2,\dotsc,R_r)$ iteratively by the following:\\
(1) If $\la$ is of rectangular shape, we set $R(\la)=(R_1)=(\la)$;\\
(2) Otherwise, define $R(\la)=((\la_1^{l(\la)}), R_2,R_3,\dotsc,R_r)$,
where $(R_2,R_3,\dotsc,R_r)$ is the rectangular filtration of $\mathfrak{C}(\la)$.
\end{definition}
We see that the $r$ in the definition is equal to the corner number of $\la$.
Now we can state the main result of this section.
\begin{theorem}\label{T:MacFiltration}
Let $R(\la)=(R_1,R_2,\dotsc,R_r)$ be the rectangular filtration of $\la$. Denoting $f_i=Q_{R_i}(q,q^\beta)$, we have
\begin{align*}\label{F:MacFiltration}
&Q_\la(q,q^\beta)\\
&=c_\la(q,q^\beta)((\cdots((f_r^*.f_{r-1})^*.f_{r-2})^*.\cdots).f_2)^*.f_1\\
                &=c_\la(q,q^\beta)\sum_{\underline{\mu}}\prod_{i=1}^rg_{R_i,\mu^i}(\beta,q)
                ((\cdots((p_{\mu^r}^*.p_{{\mu^{r-1}}})^*.p_{{\mu^{r-2}}})^*.\cdots).p_{\mu^2})^*.p_{\mu^1},
\end{align*}
where $c_\la(q,q^\beta)$ is a non-vanishing rational function of $q,q^\beta$, $g_{R_i,\mu^i}(\beta,q)$ is given by Definition \ref{D:gksmu} and the sum is over sequences of partitions $\underline{\mu}=(\mu^1,\mu^2,\dotsc,\mu^r)$ such that $|\mu^i|=|R_i|$.
\end{theorem}
\begin{proof}
 For the first equal sign, let us make induction on the corner number of $\la$. If $n_c(\la)=1$, $\la$ is already a rectangular partition and there is nothing to do. Otherwise, $Q_\la\doteq Q_{\mathfrak{C}(\la)}^*.f_1$ by (\ref{E:iterativeengine}), but $n_c(\mathfrak{C}(\la))=n_c(\la)-1$, and the induction comes in to construct $Q_{\mathfrak{C}(\la)}$. This proves the first equality.

The second equal sign follows directly from the first one by Corollary \ref{F:RecMacAsPowerSum}.
\end{proof}
We remark that for $\la=(a_1^{n_1}a_2^{n_2}\cdots a_r^{n_r})$, it is not difficult to give a direct description of its rectangular filtration
$R(\la)=(R_1,R_2,\dotsc,R_r)$:\\ $R_i=(\dot{a}_i^{n_i^*})$ where $n^*_{2i+1}=\sum_{i<j< r-i+1} n_j$,
 $n^*_{2i}=\sum_{i<j\leq r-i+1} n_j$ and $\dot{a}_{2i}=a_i-a_{r-i}$, $\dot{a}_{2i+1}=a_{i+1}-a_{r+1-i}$ ($a_{r+1}=0$).

\begin{remark}\label{R:MacAscoefficient}
We also remark that we can use this method to express Macdonald functions as the coefficients of some monomials of certain Laurent polynomial.
Let us give an example to show how to do this.

We consider partition $\la=(8,8,2,2,2)$, notice that $Q_\la\doteq Q_{(6^3)}^*.Q_{(8^5)}$.
Up to scalar, $Q_{(k^s)}(q,q^\beta)$ is the coefficient of $z_1^kz_2^k\dotsm z_s^k$ in the following
\begin{equation}
P(z_1,\dotsc,z_s)=\prod_{1\leq i<j\leq s}\Big(\frac{z_i}{z_j};q\Big)_\beta
\Big(\frac{qz_j}{z_i};q\Big)_\beta\prod_{a=1}^s\exp\Big(\sum_{n\geq1}\frac{h_{-n}z^a}{n}\frac{1-q^{n\beta}}{1-q^n}\Big).1.
\end{equation}
Thus we see that, up to a scalar multiplication, $Q_\la(q,q^\beta)$ is the coefficient of $w_1^6w_2^6w_3^6z_1^8z_2^8\dots z_5^8$ in the following:
\begin{align*}
&P(w_1,w_2,w_3)^*.P(z_1,z_2,\dotsc,z_5)\\\nonumber
&=\prod_{1\leq i<j\leq 3}\Big(\frac{w_i}{w_j};q\Big)_\beta
\Big(\frac{qw_j}{w_i};q\Big)_\beta
\prod_{1\leq i<j\leq 5}\Big(\frac{z_i}{z_j};q\Big)_\beta
\Big(\frac{qz_j}{z_i};q\Big)_\beta\\\nonumber
&\qquad\times\prod_{a=1}^3\exp\Big(\sum_{n\geq1}\frac{h_nw_a^n}{n}\frac{1-q^{n\beta}}{1-q^n}\Big)
\prod_{b=1}^5\exp\Big(\sum_{n\geq1}\frac{h_{-n}z_a^n}{n}\frac{1-q^{n\beta}}{1-q^n}\Big).1.
\end{align*}
Using the technique as in the proof of Lemma \ref{L:Xlambda},
one can remove the annihilation operators (the exponents with $h_n$'s) to the right of the creating operators (the exponents with $h_{-n}$'s).
We find that up to a scalar multiplication, $Q_{(8,8,2,2,2)}(q,q^\beta)$ is the coefficient of $w_1^6w_2^6w_3^6z_1^8z_2^8\dots z_5^8$ in
\begin{align*}
%&=\prod_{1\leq i<j\leq 3}\Big(\frac{w_i}{w_j};q\Big)_\beta
%\Big(\frac{qw_j}{w_i};q\Big)_\beta
%\prod_{1\leq i<j\leq 5}\Big(\frac{z_i}{z_j};q\Big)_\beta
%\Big(\frac{qz_j}{z_i};q\Big)_\beta\\\nonumber
%&\qquad\times\prod_{a=1}^3\prod_{b=1}^5(w_az_b;q)^{-1}_\beta
%\prod_{b=1}^5\exp\Big(\sum_{n\geq1}\frac{h_{-n}z_a^n}{n}\frac{1-q^{n\beta}}{1-q^n}\Big).1\\\nonumber
&=\prod_{1\leq i<j\leq 3}\Big(\frac{w_i}{w_j};q\Big)_\beta
\Big(\frac{qw_j}{w_i};q\Big)_\beta
\prod_{1\leq i<j\leq 5}\Big(\frac{z_i}{z_j};q\Big)_\beta
\Big(\frac{qz_j}{z_i};q\Big)_\beta\\\nonumber
&\qquad\times\prod_{a=1}^3\prod_{b=1}^5(w_az_b;q)^{-1}_\beta
\prod_{b=1}^5\sum_{n\geq0}Q_n(q,q^\beta)z_a^n.
\end{align*}
One can formalize this example to general Macdonald functions.
Then the result can be looked as a generalized form of raising operator formula,
as the raising operator formula is essentially to express a symmetric function as the coefficient of some monomial in a Laurent polynomial.
\end{remark}
\centerline{\bf Acknowledgments}
The author thanks
Professors Naihuan Jing for his help on the work. %He also acknowledges the partial support of ?
%during the work.

 \vskip 0.1in

\bibliographystyle{amsalpha}

\begin{thebibliography}{99}%{ABC}
\bibitem{A} G. E. Andrews, Problems and prospects for basic hypergeometric
functions, in ``Theory and Applications of Special Functions'' (R.
A. Askey, Ed.), Academic Press, New York, 1975.
\bibitem{BBL} H. Belbachir, A. Boussicault, J.-G. Luque, Hankel hyperdeterminants,
rectangular Jack polynomial and even powers of the Vandermonde, J.
Algebra 320 (2008), 3911--3925.

\bibitem{CJ1} W. Cai, N. Jing, On vertex operator realizations of Jack functions,
 Jour. Alg. Comb. 32, (2010), 579--595.

\bibitem{CJ2} W. Cai, N. Jing, A generalization of Newton's identity and Macdonald functions, arXiv:1210.1621.

\bibitem{CJ3} T.W. Cai, N. Jing, Jack vertex operators and realization of Jack functions,
 J. Algebr. Comb. DOI 10.1007/s10801-013-0438-9.
\bibitem{D} F. J. Dyson, Statistical theory of the energy levels of complex systems I, J. Math.
Phys. 3 (1962), 140--156.

\bibitem{FF} B. Feigin, E. Feigin, Principal subspace for the bosonic vertex operator
$\phi_{\sqrt{2m}}(z)$ and Jack polynomials, Adv.
Math. 206 (2006), no. 2, 307--328.

\bibitem{JJ} N. Jing, T. J\'ozefiak, {\em A formula for two row Macdonald
functions}, Duke Math. J. 67, No. 2 (1992), 377--385.

\bibitem{J2} N. Jing, Vertex operators and Hall-Littlewood symmetric functions,
Adv. Math. 87 (1991), no. 2, 226--248.
%
\bibitem {K} K. W. J. Kadell A Dyson Constant Term Orthogonality Relation,
J. Combin Theory Ser. A, 89, (2000), 291--297.

\bibitem{LS} M. Lassalle, M.~J. Schlosser, Inversion of the Pieri formula for Macdonald polynomials, Adv. Math. 202 (2006), 289--325.

\bibitem{KLW} G. K\'arolyi, A. Lasscoux, S. Warnaar,
 Constant term identities and Poincare polynomials, arXiv:1209.0855.

\bibitem{M} I. G. Macdonald, Symmetric functions and Hall polynomials,
 2nd ed., With contributions by A. Zelevinsky. Oxford Univ. Press, New York, 1995.

\bibitem{Ma} S. Matsumoto Hyperdeterminantal expressions for Jack
functions of rectangular shapes, J. Algebra 320 (2008), 612--632.
\bibitem{Z} D. Zeilberger and D. M. Bressoud, A proof of Andrews' q-Dyson
conjecture, Discrete Math. 54 (1985), 201--224.
\end{thebibliography}

\end{document}